\documentclass[12pt]{amsart}

\usepackage{fullpage,url,amssymb,enumerate,colonequals}
\usepackage{mathrsfs}
\usepackage[section]{placeins}
\usepackage{MnSymbol}
\usepackage{extarrows}
\usepackage{lscape}
\usepackage[all,cmtip]{xy}
\usepackage{tablefootnote}
\usepackage{footnote}

\usepackage{makecell}

\usepackage[OT2,T1]{fontenc}
\usepackage{color}

\usepackage[
        colorlinks, citecolor=darkgreen,
        backref,
        pdfauthor={Nuno Freitas},
]{hyperref}
\usepackage[capitalize,noabbrev]{cleveref}

\usepackage{comment}
\usepackage{multirow}
\usepackage{enumitem}
\usepackage{aliascnt}

\numberwithin{equation}{section}
\newtheorem{theorem}{Theorem}[section]

\newaliascnt{proposition}{theorem}
\newtheorem{proposition}[proposition]{Proposition}
\aliascntresetthe{proposition}

\newaliascnt{lemma}{theorem}
\newtheorem{lemma}[lemma]{Lemma}
\aliascntresetthe{lemma}

\newaliascnt{corollary}{theorem}
\newtheorem{corollary}[corollary]{Corollary}
\aliascntresetthe{corollary}

\theoremstyle{definition}
\newaliascnt{definition}{theorem}

\aliascntresetthe{definition}

\theoremstyle{remark}
\newaliascnt{remark}{theorem}
\newtheorem{remark}[remark]{Remark}
\aliascntresetthe{remark}

\crefname{theorem}{theorem}{theorems}
\crefname{proposition}{proposition}{propositions}
\crefname{lemma}{lemma}{lemmas}
\crefname{corollary}{corollary}{corollaries}
\crefname{definition}{definition}{definitions}
\crefname{remark}{remark}{remarks}

\definecolor{darkgreen}{rgb}{0,0.5,0}
\definecolor{rem}{rgb}{0.8,0,0}
\definecolor{new}{rgb}{0.3,0.1,0.9}
\definecolor{reply}{rgb}{0,0,0.8}
\definecolor{gray}{gray}{0.7}

\renewcommand{\gcd}{\text{gcd}}
\newcommand{\F}{\mathbb{F}}
\newcommand{\PP}{\mathbb{P}}
\newcommand{\Q}{\mathbb{Q}}

\newcommand{\Z}{\mathbb{Z}}
\newcommand{\rhobar}{{\overline{\rho}}}
\newcommand{\eps}{\varepsilon}

\newcommand{\Qbar}{{\overline{\Q}}}

\DeclareMathOperator{\Frob}{Frob}

\DeclareMathOperator{\tr}{tr}

\DeclareMathOperator{\Norm}{ Norm}

\newcommand{\cL}{\mathcal{L}}

\newcommand{\cF}{\mathcal{F}}

\newcommand{\cT}{\mathcal{T}}

\newcommand{\fq}{\mathfrak q}
\newcommand{\fQ}{\mathfrak Q}

\newcommand{\bk}{\mathbf k}

\hypersetup{
  pdftitle={On Fermat-type equations of signature (r,r,p)},
  pdfauthor={Nuno Freitas},
  pdfsubject={Hilbert modular method, exact Frobenius traces, and Fermat-type equations}
}


\begin{document}

\title{On Fermat-type equations of signature $(r,r,p)$}
\author{Nuno Freitas}
\address{Instituto de Ciencias Matemáticas (ICMAT),
          Nicolás Cabrera 13-15
         28049 Madrid, Spain}
\email{nunob.freitas@icmat.es}

\begin{abstract}
Let $r\geq11$ be a prime. We show that there are infinitely many integers $C$ for which the Fermat-type equation
$$
x^r+y^r=Cz^p
$$
has no non-trivial primitive solutions for all sufficiently large (in terms of $r$ and~$C$) prime exponents~$p$. The proof uses several Frey curves to force simultaneous Frobenius trace equalities at the primes above~$3$; a new trace separation argument shows that these equalities imply $3\mid x+y$, after which a further Frey curve and level lowering give a contradiction with the Ramanujan bound.
\end{abstract}

\maketitle

\section{Introduction}

We consider the Fermat-type equation
\begin{equation}\label{eq:rrp}
        x^r+y^r=Cz^p,
\end{equation}
where $r$ and $p$ are primes and $C$ is a positive integer. A solution
$(a,b,c)\in\Z^3$ is \emph{primitive} if $\gcd(a,b)=1$ and \emph{non-trivial}
if $abc\neq0$. It follows from a result of Darmon--Granville
\cite{DarmonGranville} that, for fixed $r$ and $p$, there are only finitely
many primitive solutions to~\eqref{eq:rrp}. It is conjectured, however, that
the finiteness of such solutions also holds when~$p$ varies.

The modular approach, originating in the proof of Fermat's Last Theorem~\cite{Wiles}, has
proved particularly effective in the study of generalized Fermat equations.
A systematic approach to~\eqref{eq:rrp} using the modular method was
developed by the author in~\cite{FreitasRecipes}, using Frey elliptic curves
over the maximal totally real subfield
$K_r=\Q(\zeta_r)^+ \subset \Q(\zeta_r)$.

Several subsequent works obtain asymptotic results for particular values of
$r$ and $C$, or under local conditions on the putative solution. Recent
examples include the work of Mocanu~\cite{Mocanu}, obtaining asymptotic
results for many primes $r<150$ and $C=1$ under the assumption that $c$ is
even; the work of Freitas--Najman~\cite{FreitasNajman}, treating arbitrary
$r$ and infinitely many coefficients, but only for a positive density set of
prime exponents~$p$ and for solutions satisfying $2\mid a+b$ or
$r\mid a+b$; 
and the more recent work of
Kara--Mocanu--\"Ozman~\cite{KaraMocanuOzman}, where, assuming the Weak
Frey--Mazur conjecture and also the Eichler--Shimura conjecture over $K_r$
when $[K_r:\Q]$ is even, they prove that~\eqref{eq:rrp} admits no
non-trivial primitive solution with $|c|>1$ for every prime $r\geq5$
and every $C\geq1$ with $r\nmid C$, provided that $p$ is sufficiently
large in terms of $r$ and~$C$. Our main theorem gives an unconditional
result of this kind for infinitely many coefficients~$C$ for every
prime $r\geq11$, and also excludes solutions with $c=\pm1$.

\begin{theorem}\label{thm:main}
Let $r\geq11$ be prime. There exists a finite set $\cL_r$ of primes
$\ell$ satisfying
\[
\ell\neq2,3,r
\qquad\text{and}\qquad
\ell\not\equiv1\pmod r
\]
such that the following holds. Let $C>2$ be an integer such that
$3\nmid C$, every prime divisor $q$ of $C$ satisfies $q\neq r$ and
$q\not\equiv1\pmod r$, and every $\ell\in\cL_r$ divides~$C$.
Then, for every sufficiently large
prime~$p$, the equation~\eqref{eq:rrp} has no non-trivial primitive
solutions; here the lower bound on~$p$ may depend on $r$ and~$C$.
\end{theorem}
To our knowledge, this is the first unconditional
asymptotic result for generalized Fermat equations of signature $(r,r,p)$ in
which both prime exponents vary over unbounded ranges, without imposing local
conditions on the putative solution.

Let us summarize the proof. It extensively uses the multi-Frey technique~\cite{BCDFMultiFrey} with
the Frey curves~$E^\bk_{a,b}$ constructed in~\cite{FreitasRecipes}, indexed by certain triples~$\bk$. The
mod~$p$ representation of $E^\bk_{a,b}$ is level lowered to
one of finitely many Hilbert newforms~$f$. Comparing the $q$-adic representations of~$f$ with the $q$-adic Tate module of the Frey curve
$E^\bk_{1,-1}$ associated with the trivial solution $(1,-1,0)$ allows us to
construct the finite set $\cL_r$. The condition that every
$\ell\in\cL_r$ divides~$C$ then forces
$a_\fq(E^\bk_{a,b})=a_\fq(E^\bk_{1,-1})
$
for the required triples~$\bk$ and all $\fq\mid3$ of~$K_r$.

The main new ingredient is a trace separation argument at the primes above~$3$.
There are only four possible values of $[a:b]$ in $\PP^1(\F_3)$. We show
that, for each value different 
from~$[1:-1]$, we can choose one of the Frey
curves $E^\bk_{a,b}$ for which its Frobenius trace differs from that of~$E^\bk_{1,-1}$ at some prime above~$3$. Hence the simultaneous trace
equalities obtained above force $[a:b]=[1:-1]$, and therefore $3\mid a+b$.
We then pass to a second Frey 
curve~$F_{a,b}$, which has multiplicative
reduction at the primes $\fq\mid3$, while these primes do not divide the
Serre level of its mod~$p$ representation. Thus level lowering mod~$p$
removes the primes above~$3$, and the resulting congruence with a Hilbert
newform contradicts the Ramanujan bound when~$p$ is sufficiently large.

Throughout this paper, $r\geq11$ is prime,
$K=K_r=\Q(\zeta_r)^+$ and $n=[K:\Q]=(r-1)/2$.

\subsection{Acknowledgements} 
This work was completed over several visits to the Max Planck Institute for Mathematics in Bonn and the author appreciates its hospitality and financial support.
\vspace{-3mm}

\subsection{AI use disclosure.}
AI models were used in the work leading to this paper. OpenAI’s GPT-5.6 Sol and GPT-6 Astra
were used for mathematical interactions, exploration and auditing. All statements and proofs were checked by the author, who takes full responsibility for the content.

\section{The Frey curves} \label{sec:Frey}

Let $r\geq11$ be a prime number, and fix a primitive $r$-th root of unity
$\zeta_r$. Put $K=\Q(\zeta_r)^+$, $n=[K:\Q]=(r-1)/2$, and let
$\fq_r$ be the unique prime of $K$ above~$r$. For $k\geq0$, we set
\[
 \omega_k=\zeta_r^k+\zeta_r^{-k},\qquad
 f_k(x,y)=x^2+\omega_kxy+y^2.
\]
The following factorization holds over $K$:
\begin{equation}\label{eq:cyclotomic-factorization}
        x^r+y^r=(x+y)\prod_{k=1}^{n}f_k(x,y).
\end{equation}

Let $\bk=(k_1,k_2,k_3)$ satisfy
\[
0\leq k_1<k_2<k_3\leq r-1,
\quad \text{ and } \quad
k_i\not\equiv\pm k_j\pmod r \;\text{ for } i\neq j
\]
and set
\[
 \alpha=\omega_{k_3}-\omega_{k_2},\qquad
 \beta=\omega_{k_1}-\omega_{k_3},\qquad
 \gamma=\omega_{k_2}-\omega_{k_1}.
\]
For integers $a,b$, not both zero, with $a+b\neq0$ when $k_1=0$, set
\[
 A_{a,b}=\alpha f_{k_1}(a,b),\qquad
 B_{a,b}=\beta f_{k_2}(a,b),\qquad
 C_{a,b}=\gamma f_{k_3}(a,b).
\]
We have $A_{a,b}+B_{a,b}+C_{a,b}=0$, and define
\begin{equation}\label{eq:Frey-general}
 Z_{a,b}^{\bk}:Y^2=X(X-A_{a,b})(X+B_{a,b}),
 \qquad
 \Delta(Z_{a,b}^{\bk})=2^4(A_{a,b}B_{a,b}C_{a,b})^2.
\end{equation}
We use the notation
\begin{equation}\label{eq:second-curve}
 F_{a,b}:=Z_{a,b}^{(0,1,2)}\quad \text{ and } \quad
 E^\bk_{a,b}:=Z^\bk_{a,b}\quad\text{when }k_1\geq1.
\end{equation}
After this section, the curve $F_{a,b}$ will only be used in
Section~\ref{sec:trace-reduction}. From now on, by a triple
$\bk=(k_1,k_2,k_3)$ we mean one satisfying
\begin{equation}\label{eq:triple-conditions}
1\leq k_1<k_2<k_3\leq r-1,
\qquad
k_i\not\equiv\pm k_j\pmod r\quad\text{for }i\neq j.
\end{equation}
\begin{remark}\label{rem:triple-symmetry}
Since $\omega_{-j}=\omega_j$ and $f_{-j}=f_j$, replacing any index by its
negative modulo~$r$ does not change the coefficients of the curve. Moreover,
permuting the three indices gives a $K$-isomorphic curve. Thus every
$E^\bk_{a,b}$ is $K$-isomorphic to one indexed by a unique triple satisfying
$1\leq k_1<k_2<k_3\leq n$.
\end{remark}

The following lemma shows that the triples needed in the proofs 
in \S\ref{sec:t=1}--\S\ref{sec:t=0} satisfy~\eqref{eq:triple-conditions}.

\begin{lemma}\label{lem:valid-triples}
Let $r\geq11$ be prime. Let $\bk$ be a triple of one of the following
forms:
\begin{enumerate}[label=\textup{(\roman*)}]
\item $\bk$ is one of
$(1,2,4),(1,3,5),(2,4,8),
(2,6,10),(1,5,7)$;
\item $\bk$ is obtained by reducing $j,2j,4j$ modulo~$r$ to
$\{1,\ldots,r-1\}$ and reordering increasingly, where
$j\not\equiv0\pmod r$;
\item $\bk$ is an increasingly ordered subset of
$\{b-4,b-2,b,b+2\}$ with cardinality three,
where~$5\leq b\leq n$.
\end{enumerate}
Then $\bk$ satisfies~\eqref{eq:triple-conditions}.
\end{lemma}
\begin{proof}
For \textup{(i)}, a congruence modulo sign between two indices would force
$r$ to divide one of $1,2,3,4,5,6,8,10,12,16$, which is impossible for a
prime $r\geq11$. Thus \eqref{eq:triple-conditions} holds.

For \textup{(ii)}, the three classes are nonzero. If two were congruent
modulo sign, then, after cancelling $j$, two of $1,2,4$ would be congruent
modulo sign. This would force $r$ to divide one of $1,2,3,5,6$, again
impossible. Thus \eqref{eq:triple-conditions} holds.

For \textup{(iii)}, we have
$1\leq b-4<b-2<b<b+2\leq n+2\leq r-1$,
so all four indices lie in the required range. Let $x\neq y$ be two
elements of $\{b-4,b-2,b,b+2\}$. Then
$|x-y|\in\{2,4,6\}$,
therefore $x\not\equiv y\pmod r$, since $r\geq11$. It remains to exclude
$x\equiv-y\pmod r$, in which case $r\mid x+y$. On the other hand,
$x+y$ is a positive even integer, and
\[
x+y\leq b+(b+2)=2b+2\leq2n+2=r+1<2r,
\]
so $r\mid x+y$ forces $x+y=r$. This is impossible,
because $x+y$ is even and $r$ is odd.
Therefore, $x\not\equiv\pm y\pmod r$ and
\eqref{eq:triple-conditions} holds.
\end{proof}

In the statements concerning solutions of~\eqref{eq:rrp} in
Sections~2--4, we assume that $C>2$ and that every prime divisor $q$
of $C$ satisfies $q\neq r$ and $q\not\equiv1\pmod r$.
Let $h_r$ denote the class number of $K$. We will always take $p>h_r$,
as required in~\cite[\S2]{FreitasRecipes}. Since $C>2$, the additional
convention $|abc|\neq1$ used there causes no conflict with our
definition of a non-trivial solution.

\begin{proposition}\label{prop:first-modularity}
There is a constant $M_r$, depending only on~$r$, such that the following
holds.  Let $(a,b,c)$ be a primitive
solution of \eqref{eq:rrp} with prime exponent $p>M_r$. Then,
\begin{enumerate}[label=\textup{(\roman*)},leftmargin=2.2em]
\item the representation $\rhobar_{E^\bk_{a,b},p}$ is absolutely irreducible and its
Serre level~$N(\rhobar_{E^\bk_{a,b},p})$ belongs to the finite set
\[
 \left\{
   \left(\prod_{\mathfrak P\mid2}\mathfrak P^{s_{\mathfrak P}}\right)
   \fq_r^{\,t}:
   s_{\mathfrak P}\in\{2,3,4\},\ t\in\{0,2\}
 \right\};
\]
\item there is a Hilbert newform $f$ over~$K$ of
parallel weight~$2$, trivial character, level~$N(\rhobar_{E^\bk_{a,b},p})$, and a prime $\lambda\mid p$ in its
field of coefficients $\Q_f$, such that
$\rhobar_{E^\bk_{a,b},p}\simeq\rhobar_{f,\lambda}$;
\item if $q\neq2,r$ is a prime with $q\not\equiv1\pmod r$, then
$E^\bk_{a,b}$ has good reduction at every prime $\fq\mid q$ in~$K$.
Moreover, $E^\bk_{a,b}$ has good reduction at~$\fq_r$ if $r\mid a+b$.
\end{enumerate}
\end{proposition}
\begin{proof}
Modularity follows from
\cite[Theorem~3.6 and Corollary~6.4]{FreitasRecipes}. Since $K$ is totally real and $\rhobar_{E^\bk_{a,b},p}$ is odd, its irreducibility is equivalent to absolute
irreducibility. Thus
absolute irreducibility for all sufficiently large~$p$ follows from
\cite[Theorems~1--2]{FreitasSiksek}, applied as in
\cite[Theorem~7.1]{FreitasRecipes}.
The conductor and Serre level are given in
\cite[Proposition~3.2 and Corollary~3.4]{FreitasRecipes}, and level
lowering is carried out in
\cite[Section~3.2]{FreitasRecipes}.
\end{proof}

\begin{remark}\label{rem:goodReduction}
A closer look at the proof of
\cite[Proposition~3.2]{FreitasRecipes} shows
that part~\textup{(iii)}
holds for all coprime integers $a,b$, without the requirement that they arise
from a solution of~\eqref{eq:rrp}. Moreover, the curve $E^\bk_{1,-1}$ has
good reduction at~$\fq_r$ and at every prime $\fq\nmid2r$ in $K$.
\end{remark}
\begin{proposition}\label{prop:second-modularity}
There is a constant $M_{r,C}$, depending only on~$r$ and~$C$, such that the following
holds.  Let $(a,b,c)$ be a non-trivial primitive
solution of \eqref{eq:rrp} with prime exponent $p>M_{r,C}$. Then,
\begin{enumerate}[label=\textup{(\roman*)},leftmargin=2.2em]
\item the representation $\rhobar_{F_{a,b},p}$ is absolutely irreducible and its
Serre level~$N(\rhobar_{F_{a,b},p})$ belongs to a finite set of levels supported at $2rC$;
\item there is a Hilbert newform $f$ over~$K$ of
parallel weight~$2$, trivial character, level~$N(\rhobar_{F_{a,b},p})$, and a prime $\lambda\mid p$ in its
field of coefficients $\Q_f$, such that
$\rhobar_{F_{a,b},p}\simeq\rhobar_{f,\lambda}$;
\item if $3\mid a+b$ and $3\nmid C$, then every prime $\fq\mid3$ in $K$ is a prime of multiplicative reduction for $F_{a,b}$ and
$\fq$ does not divide $N(\rhobar_{F_{a,b},p})$.
\end{enumerate}
\end{proposition}
\begin{proof}
Modularity follows from
\cite[Theorem~4.4 and Corollary~6.4]{FreitasRecipes}.
Absolute irreducibility for all sufficiently large~$p$ follows from
\cite[Theorems~1--2]{FreitasSiksek}, applied as in
\cite[Theorem~7.1]{FreitasRecipes}.
The conductor and Serre level are given in
\cite[Propositions~4.1 and~4.2]{FreitasRecipes}, and level lowering is
carried out in \cite[Section~4.2]{FreitasRecipes}.
The final assertion follows from
\cite[Propositions~4.1 and~4.2]{FreitasRecipes}.
\end{proof}

\section{Forcing equality of traces at 3}

We will need the following lemma in representation theory.

\begin{lemma}\label{lem:coset}
Let $G$ be a group, $\chi:G\to\{\pm1\}$ a nontrivial character, and $F$ an algebraically closed field of characteristic 0.
Let $V$ and $U$ be 2-dimensional representations over $F$ with~$V$ irreducible and $U$ semisimple.
Suppose that
$V\not\simeq V\otimes\chi$ and that
\[
        \tr U(g)=\tr V(g)
\]
whenever $\chi(g)=-1$. Then $U\simeq V$.

Moreover, for continuous $\ell$-adic representations of the absolute Galois group of
a number field, it is enough to assume the trace equality at all but finitely
many unramified Frobenius elements satisfying $\chi(\Frob_\fq)=-1$.
\end{lemma}

\begin{proof}
Consider the semisimple representations
$$
 \rho_1=V\oplus(U\otimes\chi)
 \quad\text{and}\quad
 \rho_2=U\oplus(V\otimes\chi).
$$

If $\chi(g)=1$, then $\tr\rho_1(g)=\tr\rho_2(g)$, while if
$\chi(g)=-1$, then $\tr\rho_1(g)=0=\tr\rho_2(g)$ by hypothesis.
Thus $\rho_1$ and $\rho_2$ have the same character and hence are
isomorphic by Brauer--Nesbitt.
Since $V$ is irreducible, it must occur on the right either in~$U$ or
as $V\otimes\chi$. The latter is excluded by hypothesis. Thus $V$
occurs in~$U$, and since both representations have dimension~$2$, we
obtain $U\simeq V$.

For the last statement, the Chebotarev density theorem and continuity show that the required trace equality holds for all elements $g$ satisfying $\chi(g)=-1$, and the first part applies.
\end{proof}

\begin{lemma}\label{lem:trivial-reduction}
Let $q\neq2,r$ be a prime such that $q\not\equiv1\pmod r$. Let
$\fq\mid q$ be a prime in $K$ and $a,b\in\Z$ be coprime. If $q\mid a+b$,
then, for every triple~$\bk$, the reductions of $E^\bk_{a,b}$ and
$E^\bk_{1,-1}$ are isomorphic over $\F_\fq$. In particular,
$a_\fq(E^\bk_{a,b})=a_\fq(E^\bk_{1,-1})$.
\end{lemma}
\begin{proof}
By Proposition~\ref{prop:first-modularity}\textup{(iii)} and
Remark~\ref{rem:goodReduction}, both curves have good reduction at~$\fq$.
Since $q\mid a+b$ and $a,b$ are coprime, $q\nmid ab$. We have
$a\equiv-b\pmod\fq$, therefore
\[
 f_k(a,b)\equiv(2-\omega_k)b^2\pmod\fq,
\]
and it follows that
\[
 A_{a,b}\equiv b^2A_{1,-1} \pmod\fq,
 \qquad
 B_{a,b}\equiv b^2B_{1,-1}\pmod\fq.
\]
Thus the change of variables $X=b^2X'$ and $Y=b^3Y'$ identifies their
reductions over~$\F_\fq$.
\end{proof}

We can now prove the main result of this section.

\begin{proposition}\label{prop:finite-L}
Fix $r \geq 11$ a prime and a triple $\bk$. There exist a finite set $\cL_\bk$ of rational
primes and a constant $M_\bk$ such that:
\begin{enumerate}[label=\textup{(\roman*)},leftmargin=2.2em]
\item every $\ell\in\cL_\bk$ satisfies
$\ell\neq2,3,r$ and $\ell\not\equiv1\pmod r$;
\item if $(a,b,c)$ is a non-trivial primitive solution of~\eqref{eq:rrp} with exponent
$p>M_\bk$ and every $\ell\in\cL_\bk$ divides $a+b$, then
$$
        a_\fq(E^\bk_{a,b})=a_\fq(E^\bk_{1,-1})
        \qquad\text{for every }\fq\mid3.
$$
\end{enumerate}
\end{proposition}
\begin{proof}
Let $\eps:G_K\to\{\pm1\}$ be the character corresponding to the extension $\Q(\zeta_r)/K$.

Fix a prime $q\nmid6r$ and let $V$ be the $G_K$ representation arising on the $q$-adic Tate module of $E^\bk_{1,-1}$ with scalars extended to~$\overline{\Q}_q$. The representation $V$ is absolutely irreducible. Moreover, $V\not\simeq V\otimes\eps$: the representation $V$ is unramified at~$\fq_r$ ($E^\bk_{1,-1}$ has good reduction at~$\fq_r$ by \Cref{rem:goodReduction}), whereas $\eps$ ramifies at~$\fq_r$.

Denote by~$\cF_\bk$ the finite set of newforms that can
occur after level lowering as described by \Cref{prop:first-modularity}.
For each $f\in\cF_\bk$, take a prime $\fQ \mid q$ in $\Q_f$ and let $U_f$ be its $\fQ$-adic representation extended to~$\Qbar_q$. Suppose that $U_f\not\simeq V$.

By \Cref{lem:coset}, there are infinitely many primes $\fq\nmid6rqN_fN_{E^\bk_{1,-1}}$ such that $\eps(\Frob_\fq)=-1$ and $a_\fq(f)\neq a_\fq(E^\bk_{1,-1})$. Choose one such prime $\fq_f$ and let $\ell_f$ be the rational prime below it. Then $\ell_f\not\equiv1\pmod r$, since a prime congruent to~$1$ modulo~$r$ splits completely in~$\Q(\zeta_r)$. We have $D_f:=\Norm_{\Q_f/\Q}(a_{\fq_f}(f)-a_{\fq_f}(E^\bk_{1,-1}))\in\Z_{\neq 0}$.

Let $\cL_\bk$ be the set of the primes $\ell_f$ obtained in this way. Choose $M_\bk$ larger than the bound in \Cref{prop:first-modularity}, every $\ell_f$, every prime divisor of all the~$D_f$, and $4\sqrt{\Norm\fq}$ for every $\fq\mid3$.

Now let $(a,b,c)$ satisfy the hypotheses of the proposition, and let $f$ and $\lambda\mid p$ be given by \Cref{prop:first-modularity}, so that $\rhobar_{f,\lambda}\simeq\rhobar_{E^\bk_{a,b},p}$. Suppose that $U_f\not\simeq V$. Then $\ell_f\in\cL_\bk$, and hence $\ell_f\mid a+b$ by assumption. Thus \Cref{lem:trivial-reduction} gives $a_{\fq_f}(E^\bk_{a,b})=a_{\fq_f}(E^\bk_{1,-1})$. The residual isomorphism implies  $\lambda\mid a_{\fq_f}(f)-a_{\fq_f}(E^\bk_{1,-1})$ and taking norms gives $p\mid D_f$, a contradiction. Thus $U_f\simeq V$.

Therefore, for every $\fq\mid3$, we have $a_\fq(f)=a_\fq(E^\bk_{1,-1})$ in~$\Q_f$, and again by the residual isomorphism we obtain $a_\fq(E^\bk_{a,b})\equiv a_\fq(E^\bk_{1,-1})\pmod p$. The conclusion now follows from the Hasse bound and the choice of~$M_\bk$.
\end{proof}

\begin{corollary}\label{cor:simultaneous-forcing}
Let $\cT$ be a finite set of triples. There exist a finite set $\cL_\cT$ of prime numbers and a constant $M_\cT$ such that, if $(a,b,c)$ is a non-trivial primitive
solution of~\eqref{eq:rrp} with exponent $p>M_\cT$ and every
$\ell\in\cL_\cT$ divides $a+b$, then
\[
 a_\fq(E^\bk_{a,b})=a_\fq(E^\bk_{1,-1}) \quad \text{ for all } \;
  \bk\in\cT \text{ and all } \fq\mid3.
\]
Every prime in $\cL_\cT$ is different from $2,3,r$ and is not congruent to
$1$ modulo~$r$.
\end{corollary}

\begin{proof}
Take the union of the sets $\cL_\bk$ and the maximum of the corresponding
bounds in \Cref{prop:finite-L}.
\end{proof}

\section{Reduction to trace separation at 3}\label{sec:trace-reduction}

In the previous section, we proved that we can force equality of traces at primes $\fq \mid 3$ of the Frey curve $E^\bk_{a,b}$ and the obstructing curve $E^\bk_{1,-1}$ associated with the trivial solution $(1,-1,0)$.
In this section, using the multi-Frey technique with both $E^\bk_{a,b}$ and $F_{a,b}$, we will reduce the proof of Theorem~\ref{thm:main} to the problem of distinguishing traces at primes $\fq \mid 3$.

Note that for $[v:w]\in\PP^1(\F_3)$ the curve $E^\bk_{v,w}$ over $\F_\fq$ is an
elliptic curve by \Cref{rem:goodReduction}.

\begin{theorem}\label{thm:local-criterion}
Let $r\geq11$ be prime. Suppose there is a finite set $\cT$ of triples such
that, for every $[v:w]\in\PP^1(\F_3)$ different from $[1:-1]$, there exist
$\bk\in\cT$ and $\fq\mid3$ in $K$ satisfying
\[
        a_\fq(E^\bk_{v,w})\neq a_\fq(E^\bk_{1,-1}).
\]
Then the conclusion of \Cref{thm:main} holds for~$r$.
\end{theorem}

\begin{proof}
Let $\cT$ be as in the statement and let $M_\cT$ and $\cL_\cT$ be given by
\Cref{cor:simultaneous-forcing}. Set $\cL_r=\cL_\cT$.
Let $C>2$ be an integer with $3\nmid C$ such that every prime divisor
$q$ of $C$ satisfies $q\neq r$ and $q\not\equiv1\pmod r$, and suppose
that every prime in~$\cL_r$ divides~$C$. Denote
by~$\cF_{r,C}$ the finite set of newforms that can occur after level lowering
as described by \Cref{prop:second-modularity}.
Choose a prime $\fq_0\mid3$.  Since
$3\nmid C$, the prime $\fq_0$ does not divide the level of any form in~$\cF_{r,C}$. Thus, for
$g\in\cF_{r,C}$, we can define the integer
\begin{equation}\label{eq:second-newform-norm}
 A_g:=\Norm_{\Q_g/\Q}\bigl(a_{\fq_0}(g)^2-(\Norm \fq_0+1)^2\bigr).
\end{equation}
We have $A_g \neq 0$ by the Ramanujan bound
$|\sigma(a_{\fq_0}(g))|\leq2\sqrt{\Norm \fq_0}<\Norm \fq_0+1$ where $\sigma$ is any complex embedding (see~\cite{BlasiusRamanujan}). 
Choose $B_{r,C}$ larger than $M_\cT$, the bound in
\Cref{prop:second-modularity}, every prime divisor of~$C$, and every prime divisor of~$A_g$ for $g\in\cF_{r,C}$. 

Now let $p>B_{r,C}$ be prime and suppose that $(a,b,c)$ is a non-trivial primitive solution of~\eqref{eq:rrp} with exponent~$p$ and coefficient~$C$.

Every prime $q\mid C$ satisfies $q \not\equiv 1 \pmod r$, hence it divides $a+b$. In particular, every prime in~$\cL_\cT$ divides $a+b$, so \Cref{cor:simultaneous-forcing} gives
\[
 a_\fq(E^\bk_{a,b})=a_\fq(E^\bk_{1,-1}) \quad \text{ for all } \;
  \bk\in\cT \text{ and all } \fq\mid3.
\]
Let $[v:w]$ be the reduction of $[a:b]$ in $\PP^1(\F_3)$. The reduction of
$E^\bk_{a,b}$ is isomorphic to $E^\bk_{v,w}$, so $[v:w]=[1:-1]$ by our assumption. We conclude that $3\mid a+b$.

We now switch Frey curves.
By \Cref{prop:second-modularity} (iii), the curve $F_{a,b}$ has multiplicative reduction at
$\fq_0$, the prime $\fq_0$ does not divide $N(\rhobar_{F_{a,b},p})$, and there exist
$g\in\cF_{r,C}$ and $\lambda\mid p$ in~$\Q_g$ such that
$\rhobar_{F_{a,b},p}\simeq\rhobar_{g,\lambda}$. This means that level lowering modulo~$p$ occurs at $\fq_0$, therefore $a_{\fq_0}(g)\equiv\pm( \Norm \fq_0 +1)\pmod\lambda$  
and so $p \mid A_g$, a contradiction.
\end{proof}
For a triple $\bk$, we set
\[
 A_t=\alpha(1+\omega_{k_1}t+t^2),\qquad
 B_t=\beta(1+\omega_{k_2}t+t^2),\qquad
 C_t=\gamma(1+\omega_{k_3}t+t^2).\]
We have $A_t+B_t+C_t=0$ and define
\begin{equation}\label{eq:Et}
E^\bk_t : y^2 = x^3+ (B_t - A_t)x^2-A_tB_tx \quad \text{ and } \quad h_t:= B_t - A_t.
\end{equation}
The Frey curve $E^\bk_{a,b}$ reduces modulo a prime $\fq \mid 3$ to $E^\bk_{v,w}$ over $k = \F_\fq$ with $[v:w] \in \PP^1(\F_3)$.
After the change of variables $t=w/v$ when $v \neq 0$, the curve $E^\bk_{v,w}$ becomes $E^\bk_t$. Since $f_k(1,0)=f_k(0,1)=1$, 
when $v=0$,
we obtain
$E^\bk_{0,w} = E^{\bk}_\infty=E^{\bk}_0$.
Therefore, we have the four parameters $t=-1,0,1,\infty$. In view of
Theorem~\ref{thm:local-criterion}, we must separate the traces at $t=0,1$
from the trace at $t=-1$.
The rest of this paper is concerned with this separation.

\section{The trace of Frobenius at 3 modulo~$3$}

To understand these traces, we first study their reduction modulo~$3$. 

Let $q = 3^f$ and $k = \F_q$. 
Let $\chi:k\to\{0,\pm1\}$ be the quadratic character, defined by
\begin{equation}\label{eq:chi}
\chi(0)=0,\qquad
\chi(z)=
\begin{cases}
1,& z\in k^{\times2},\\
-1,& z\in k^\times\setminus k^{\times2}.
\end{cases}
\end{equation}

Consider an elliptic curve of the shape
$E/k:y^2=x^3+dx^2+cx$.
\begin{lemma}\label{lem:hasse-quotient}
We have
\[
a_k(E)\equiv\Norm_{k/\F_3}(d)\pmod3.
\]
\end{lemma}
\begin{proof}
Set $N=(q-1)/2$.
For each fixed $x\in k$, the number of $y\in k$ satisfying
$y^2=x^3+dx^2+cx$ is $1+\chi(x^3+dx^2+cx)$. From
$a_k(E)=q+1-\#E(k)$, we obtain
\[
a_k(E)=-\sum_{x\in k}\chi(x^3+dx^2+cx)
\]
and, after identifying $\chi(z)$ with $z^N$ in~$k$ and reducing $a_k(E)$ modulo~3, we have
\[
a_k(E) = -\sum_{x\in k}(x^3+dx^2+cx)^N \quad \text{ in } k.
\]
Note that, for $m>0$, one has
$\sum_{x\in k}x^m=-1$ if $q-1\mid m$ and
$\sum_{x\in k}x^m=0$ otherwise. Since
$3N=3(q-1)/2<2(q-1)$, the only positive multiple of $q-1$ occurring
as an exponent after expanding $(x^3+dx^2+cx)^N$ is $q-1$.
Thus, $a_k(E)$ is the coefficient of $x^{q-1}$ in the previous polynomial sum. Moreover, writing
\[
(x^3+dx^2+cx)^N=\sum_m c_mx^m,
\]
we obtain
\[
-\sum_{x\in k}(x^3+dx^2+cx)^N
=
-\sum_m c_m\sum_{x\in k}x^m
=
-c_{q-1}\sum_{x\in k}x^{q-1}
=
c_{q-1}.
\]
Therefore,
from $2N = q-1$ and
$(x^3+dx^2+cx)^N =x^N(x^2+dx+c)^N$,
it remains to compute the coefficient of $x^N$ in
$(x^2+dx+c)^N$. Since $N=1+3+\cdots+3^{f-1}$, in characteristic~$3$
we have
$$
(x^2+dx+c)^N
=
\prod_{i=0}^{f-1}
\left(x^{2\cdot3^i}+d^{3^i}x^{3^i}+c^{3^i}\right)
$$
and the only
way to obtain $x^N$ is to choose $d^{3^i}x^{3^i}$ from every factor. 
Thus
\[
a_k(E)\equiv \prod_{i=0}^{f-1}d^{3^i}
=\Norm_{k/\F_3}(d)\pmod3.
\]
\end{proof}
We now apply the lemma to the curve defined in~\eqref{eq:Et}, for which
$d=h_t$ and $t\in\F_3$.
 \begin{corollary}\label{cor:hasse-quotient}
If both $h_t$ and $h_{-1}$ are non-zero, then
\[
 a_k(E^{\bk}_t)\equiv a_k(E^{\bk}_{-1})\pmod3
 \quad\Longleftrightarrow\quad
 \frac{h_t}{h_{-1}}\in k^{\times2}.
\]
If exactly one of $h_t$ and $h_{-1}$ vanishes, 
then $a_k(E^{\bk}_t) \not\equiv a_k(E^{\bk}_{-1})\pmod3$.
\end{corollary}
\begin{proof}
From the lemma we have the equalities in $\F_3$
\[
 a_k(E^{\bk}_t) = \Norm_{k/\F_3}(h_t) \quad \text{ and }  \quad a_k(E^{\bk}_{-1}) = \Norm_{k/\F_3}(h_{-1}).
\]
For $x\in k^\times$, we have $\Norm_{k/\F_3}(x) = x^{(3^f-1)/2} = \chi(x)$, where $\chi$ is the quadratic character defined in~\eqref{eq:chi}.
Thus, when both $h_t$ and $h_{-1}$ are non-zero, the traces are congruent modulo~$3$
if and only if $h_t/h_{-1}$ is a square. If exactly one vanishes, then the
corresponding trace is $0$ modulo~$3$ and the other one is $\pm1$.
\end{proof}

\section{Separating traces at $t=1$ from $t=-1$}
\label{sec:t=1}

We now take a prime $\fq\mid3$ of $K$ and write
$k=\F_\fq\simeq\F_{3^f}$. Since $K/\Q$ is Galois, the integer $f$
is independent of~$\fq$ and is the least positive integer such that
$3^f\equiv\pm1\pmod r$. Let $\mu_r$ be the group of $r$-th roots of
unity in $\overline{k}$. Let $\chi$ be the quadratic character defined explicitly
in~\eqref{eq:chi}.

Throughout this section, the elements $\omega_j$ are viewed in $k$, so all the following calculations are in characteristic~$3$.
 
The section has two goals. First, assuming that the traces at $t=0$ and
$t=-1$ agree for certain triples, we show in
\Cref{lem:direct-preparation} that every $\omega_j$ and every
$1+\zeta_r^j$ is a square in the residue field; this information will be
used in the next section. We then treat $t=1$. Using the two identities in
\Cref{lem:two-products}, we prove in \Cref{prop:t1-direct} that simultaneous
equality of the traces at $t=1$ and $t=-1$ would force $-1$ to be a square,
giving a contradiction.

\begin{lemma}\label{lem:cyclotomic-squares}
For every $j\not\equiv0\pmod r$, the element $\omega_j-1$ is a
nonzero square in~$k$. Moreover,
\[
 \omega_j+1\in k^{\times2}
 \quad\Longleftrightarrow\quad
 3^f\equiv1\pmod r.
\]
If $f$ is even, then $3^f\equiv-1\pmod r$.
\end{lemma}

\begin{proof}
Fix $j\not\equiv0\pmod r$ and put $\zeta=\zeta_r^j$ and
$u=\omega_j=\zeta+\zeta^{-1}$. Since squaring is an automorphism of
$\mu_r$, there is a unique $w\in\mu_r$ satisfying $w^2=\zeta$.
Because we are in characteristic~$3$, we have
\[
u-1=(w+w^{-1})^2,\qquad
u+1=(w-w^{-1})^2.
\]
Thus $u-1$ is a nonzero square once $w+w^{-1}\in k$.

Since $u\in k$, the pairs $\{\zeta,\zeta^{-1}\}$ and
$\{\zeta^{3^f},\zeta^{-3^f}\}$ are both the roots of
$X^2-uX+1$. Hence $\zeta^{3^f}=\zeta$ or
$\zeta^{3^f}=\zeta^{-1}$, according as
$3^f\equiv1$ or $-1\pmod r$. As $w^{3^f}\in\mu_r$, the uniqueness
of the square root of~$\zeta$ in $\mu_r$ gives
\[
 w^{3^f}=w \quad\text{if }3^f\equiv1\pmod r,
 \qquad
 w^{3^f}=w^{-1} \quad\text{if }3^f\equiv-1\pmod r.
\]
In either case $w+w^{-1}$ is fixed by Frobenius and therefore belongs
to~$k$. On the other hand, $w-w^{-1}$ belongs to~$k$ precisely in the
first case. Since the only square roots of $u+1$ in $\overline{k}$ are
$\pm(w-w^{-1})$, this proves the first two assertions.

Finally, suppose that $f$ is even and $3^f\equiv1\pmod r$. Then
$(3^{f/2})^2\equiv1\pmod r$, so $3^{f/2}\equiv\pm1\pmod r$,
contradicting the minimality of~$f$. Therefore $3^f\equiv-1\pmod r$
when $f$ is even.
\end{proof}

\begin{lemma}\label{lem:h-relations}
For a triple $\bk=(k_1,k_2,k_3)$, write $h_t^\bk$ for the element
$h_t$  in~\eqref{eq:Et}. Then, in $k$,
\[
\begin{aligned}
h_0^\bk&=\omega_{k_1}+\omega_{k_2}+\omega_{k_3},\\
h_1^\bk&=-(\omega_{k_1}+\omega_{k_2}+\omega_{k_3})
-(\omega_{k_1}\omega_{k_2}+\omega_{k_1}\omega_{k_3}
+\omega_{k_2}\omega_{k_3}),\\
h_{-1}^\bk&=\omega_{k_1}\omega_{k_2}+\omega_{k_1}\omega_{k_3}
+\omega_{k_2}\omega_{k_3}
-(\omega_{k_1}+\omega_{k_2}+\omega_{k_3}).
\end{aligned}
\]
Moreover, no two of $h_0^\bk,h_1^\bk,h_{-1}^\bk$ vanish simultaneously.
\end{lemma}
\begin{proof}
The displayed identities follow by substituting $t=0,1,-1$ in the
definitions of $A_t$ and~$B_t$ and expanding $h_t=B_t-A_t$ in
characteristic~$3$.

If two of $h_0,h_1,h_{-1}$ vanished, the displayed formulas would give
\[
\omega_{k_1}+\omega_{k_2}+\omega_{k_3}=0,\qquad
\omega_{k_1}\omega_{k_2}+\omega_{k_1}\omega_{k_3}
+\omega_{k_2}\omega_{k_3}=0.
\]
Therefore
\[
(X-\omega_{k_1})(X-\omega_{k_2})(X-\omega_{k_3})=X^3-c
\]
for some $c\in k$. In characteristic~$3$, $X^3-c$ has at most one
distinct root, contradicting the fact that
$\omega_{k_1},\omega_{k_2},\omega_{k_3}$ are distinct.
\end{proof}
We will also use the following notations
\[
 u=\omega_1,\qquad v=\omega_2=u^2+1,\qquad d=u^2-1 = (u-1)(u+1). 
\]
\begin{lemma}\label{lem:omega-nonvanishing}
In $k$ the elements $u,v,d$ are non-zero and $\omega_j \not\in \{0,\pm 1 \}$ for all $j\not\equiv0\pmod r$.
\end{lemma}

\begin{proof}
Let $\mathfrak Q$ be a prime of $\Q(\zeta_r)$ above $\fq$, let
$k_{\mathfrak Q}$ be its residue field, and let $\bar\zeta$ denote $\zeta_r$ modulo~$\mathfrak Q$. We view $k$ as a subfield of $k_{\mathfrak Q}$.
Since $r\neq3$, for all $j\not\equiv0\pmod r$, the element
$\bar\zeta^j$ has order $r$. If $\omega_j=0$ in $k$, then in
$k_{\mathfrak Q}$, we have
$\bar\zeta^j+\bar\zeta^{-j}=0$,
which gives $\bar\zeta^{2j}=-1$, so $\bar\zeta^j$ has order dividing
$4$. If $\omega_j=1$ or $\omega_j=-1$, then
\[
\bar\zeta^{2j}-\bar\zeta^j+1=(\bar\zeta^j+1)^2=0
\quad\text{or}\quad
\bar\zeta^{2j}+\bar\zeta^j+1=(\bar\zeta^j-1)^2=0,
\]
so $\bar\zeta^j=-1$ or $\bar\zeta^j=1$, respectively. Each case
contradicts $r\geq11$. Hence
$\omega_j\notin\{0,1,-1\}$ for every $j\not\equiv0\pmod r$.
In particular, $u=\omega_1$, $v=\omega_2$ and
$d=u^2-1$ are nonzero in $k$.
\end{proof}

For $j\not\equiv0\pmod r$, let $\mathbf f_j$ denote the triple obtained
in \Cref{lem:valid-triples}\textup{(ii)} from $(j,2j,4j)$. The following lemma separates the traces at $t=0$ from the trace at $t=-1$
whenever its conclusions do not hold.

\begin{lemma}\label{lem:direct-preparation}
Let $\fq\mid3$ be a prime of $K$ and put $k=\F_\fq$. Suppose that the traces at $t=0$ and
$t=-1$ are congruent modulo~$3$ for every triple $\mathbf f_j$, with
$j\not\equiv0\pmod r$, and for the triple $(2,6,10)$. Then
$\mu_r\subset k$, $f$ is odd, and every $1+\zeta_r^j$,
$j\not\equiv0\pmod r$, is a nonzero square in~$k$.
\end{lemma}
\begin{proof}
Let $\chi$ be the quadratic character defined explicitly in~\eqref{eq:chi}. By
\Cref{lem:h-relations}, applied to $\mathbf f_j$, we have
\[
h_0^{\mathbf f_j}=\omega_j(\omega_j+1)^3,\qquad
h_{-1}^{\mathbf f_j}=(\omega_j-1)(\omega_j+1)^5.
\]
These elements are nonzero by \Cref{lem:omega-nonvanishing}. Hence
\Cref{cor:hasse-quotient}, together with the fact that $\omega_j-1$ is a
square by \Cref{lem:cyclotomic-squares}, gives $\chi(\omega_j)=1$ for every
$j\not\equiv0\pmod r$.

Applying \Cref{lem:h-relations} to $(2,6,10)$ gives
\[
h_0^{(2,6,10)}=u^2v^3d,\qquad
h_{-1}^{(2,6,10)}=uv^2\omega_7d^2.
\]
These elements are again nonzero. By the assumed trace congruence and
\Cref{cor:hasse-quotient}, their quotient is a square. Since
$u=\omega_1$, $v=\omega_2$ and $\omega_7$ are squares by the previous
paragraph, this gives $\chi(d)=1$. Thus $d=(u-1)(u+1)$ is a square.
Since $u-1$ is a square by \Cref{lem:cyclotomic-squares}, $u+1$ is also a
square. The same lemma gives $3^f\equiv1\pmod r$, hence $\mu_r\subset k$;
moreover, $f$ is odd.

The group $\mu_r(k)$ has odd order~$r$ and therefore consists of squares.
Since $\omega_j$ is a square and
$\omega_j=\zeta_r^{-j}(1+\zeta_r^{2j})$, the element
$1+\zeta_r^{2j}$ is a nonzero square. As multiplication by~$2$ permutes
the nonzero residues modulo~$r$, every $1+\zeta_r^j$ with
$j\not\equiv0\pmod r$ is a nonzero square in~$k$.
\end{proof}

We end this section with a result separating the traces at $t=1$ from
those at $t=-1$ in all cases. We need the following auxiliary lemma.
\begin{lemma}\label{lem:two-products}
For a triple~$\bk$, define
$P_\bk=h_1^\bk h_{-1}^\bk$.
Then
\[
d^5P_{(1,3,5)}=u^2P_{(2,4,8)}\quad \text{ and } \quad 
d^8P_{(2,6,10)}=-v^4P_{(1,3,5)}P_{(1,5,7)}.
\]
\end{lemma}

\begin{proof}
Using $\omega_0=2$, $\omega_1=u$ and
\[
\omega_{j+1}=u\omega_j-\omega_{j-1},
\]
together with Lemma~\ref{lem:h-relations}, we obtain
\[
\begin{aligned}
P_{(1,3,5)}
    &=-u^4d^3(u^6+u^4-1),&
P_{(2,4,8)}
    &=-u^2d^8(u^6+u^4-1),\\
P_{(2,6,10)}
    &=-u^5v^4\omega_7d^3(u^6+u^4-1),&
P_{(1,5,7)}
    &=-u\omega_7d^8.
\end{aligned}
\]
The two identities follow immediately.
\end{proof}

\begin{proposition}\label{prop:t1-direct}
For every prime $\fq\mid3$, at least one of the four triples
\[
 (1,3,5),\qquad(2,4,8),\qquad(2,6,10),\qquad(1,5,7)
\]
(possibly depending on~$\fq$) satisfies
\[
 a_\fq(E_1^\bk)\not\equiv a_\fq(E_{-1}^\bk)\pmod3.
\]
\end{proposition}
\begin{proof}
Fix $\fq\mid3$, put $k=\F_\fq$, and let $\chi$ be its quadratic character.
Suppose that
$a_\fq(E_1^\bk)\equiv a_\fq(E_{-1}^\bk)\pmod3$
for the four triples~$\bk$ in the statement.
By \Cref{lem:h-relations} and \Cref{cor:hasse-quotient}, each of the four elements
\[
P_{(1,3,5)},\quad P_{(2,4,8)},\quad
P_{(2,6,10)},\quad P_{(1,5,7)}
\]
is a nonzero square in $k$.
Applying $\chi$ to the first 
identity in \Cref{lem:two-products}, we obtain
$\chi(d)=1$,
since $u^2$ and both $P_{(1,3,5)}$ and $P_{(2,4,8)}$ are squares.
Thus $d=(u-1)(u+1)$ is a square. Since $u-1$ is also a square by
\Cref{lem:cyclotomic-squares}, it follows that $u+1$ is a square.
The same lemma then gives that
$ 
3^f\equiv1\pmod r,
$
and $f$ is odd.

Now apply~$\chi$ to the second identity
in~\Cref{lem:two-products}. Since $d^8$, $v^4$ and all three
$P_\bk$ occurring there are squares, we obtain
$\chi(-1)=1$. On the other hand, $f$ is odd, so
\[
|k|=3^f\equiv3\pmod4,
\]
and therefore $-1$ is not a square in $k$, a contradiction.
\end{proof}

\section{Separating traces at $t=0$ from $t=-1$} \label{sec:t=0}

Recall that $r \geq 11$ is a prime, 
$\fq\mid3$ is a prime in $K$ and $k=\F_\fq \simeq \F_{3^f}$. Let $\chi$ be the
quadratic character defined in~\eqref{eq:chi}. It remains to separate the
traces at $t=0$ from those at $t=-1$.

The argument in this section treats four triples simultaneously. Using the
square information obtained in the previous section, we construct in
\Cref{lem:quartet} four triples, depending on $r$ and~$\fq$, with the relevant
properties. The models in \Cref{lem:t0-models} and the character-sum identity
in \Cref{lem:four-point-count} then show that if none of these triples
separates the traces at $t=0$ and $t=-1$ modulo~$48$, the sum of the four
corresponding trace differences is congruent to~$8$ modulo~$16$. This gives
the contradiction used in the proof of \Cref{thm:trace-separation}.

For a subset $A\subset k^\times$ with three elements, set
\[
 G_A(W)=\prod_{a\in A}(W-a),
\]
consider the smooth projective models of
\[
 C_A^0:Y^2=-G_A(W),\quad \text{ and } \quad C_A^{-1}:Y^2=WG_A(W),
\]
and define
\[
 T(A): =a_k(C_A^{-1})-a_k(C_A^0).
\]
Let $a_1,a_2,a_3,a_4\in k$ be four distinct elements and, for $1\leq j \leq 4$, define 
\[
 p_j=\prod_{i\ne j}\chi(a_j-a_i).
\]
\begin{lemma}\label{lem:t0-models}
Let $\bk$ be a triple and 
$A_\bk=\{\omega_{k_1}+1,\omega_{k_2}+1,\omega_{k_3}+1\}$.
Then $E_0^\bk$ and $E_{-1}^\bk$ are isomorphic over~$k$ to
$C_{A_\bk}^0$ and $C_{A_\bk}^{-1}$, respectively.
\end{lemma}
\begin{proof}
Set $g(V)=\prod_{i=1}^3(V-\omega_{k_i})$. 

For $t=0$, the substitution
$X=\omega_{k_3}-V$ in~\eqref{eq:Et} gives
$E_0^\bk:Y^2=-g(V)$.

For $t=-1$, the change of variables
\[
 X=(\omega_{k_1}+1)(\omega_{k_2}+1)
   \frac{V-\omega_{k_3}}{V+1},\qquad
 Y_E=-\prod_{i=1}^3(\omega_{k_i}+1)\frac{Y}{(V+1)^2}
\]
identifies $E_{-1}^\bk$ with the smooth projective model of
$Y^2=(V+1)g(V)$. The factors $\omega_{k_i}+1$ are nonzero by
\Cref{lem:omega-nonvanishing}, while the differences
$\omega_{k_i}-\omega_{k_j}$ are nonzero by~\eqref{eq:triple-conditions}.

Finally, in both cases setting $W=V+1$ gives the result.
\end{proof}

Note that the next lemma is for arbitrary finite fields not only $k=\F_\fq$.

\begin{lemma}\label{lem:four-point-count}
Let $k$ be a finite field such that
$|k|\equiv3\pmod4$, and 
let $a_1,a_2,a_3,a_4 \in k$ be distinct
nonzero squares. Then
\begin{equation}\label{eq:four-point-count}
 \sum_{|A|=3}T(A)\equiv-2\sum_{j=1}^4p_j\pmod{16},
\end{equation}
where $A$ runs over all the subsets
$\{a_1,a_2,a_3,a_4\}$ with three elements.
\end{lemma}
\begin{proof}
Since $\chi(-1)=-1$, for every nonzero $a \in k^{\times2}$ the involution
$w\mapsto a^2/w$ of $k^{\times2}$ changes $\chi(w-a)$ to its negative.
Its only fixed point in $k^{\times2}$ is $w=a$, which contributes $\chi(w-a)=0$.
Hence
\begin{equation}\label{eq:singleton-sum}
 \sum_{w\in k^{\times2}}\chi(w-a)=0.
\end{equation}

The cubic $C_A^0$ has one rational point at infinity and the monic
quartic $C_A^{-1}$ has two, therefore
\[
 T(A)=-1-\sum_{w\in k}(1+\chi(w))\chi(G_A(w)).
\]
Since $\chi(G_A(0))=-1$, the term at zero cancels the initial $-1$, and
we obtain
\begin{equation}\label{eq:trace-square-sum}
 T(A)=-2\sum_{w\in k^{\times2}}
 \prod_{a\in A}\chi(w-a).
\end{equation}

For $w\in k^{\times2}$, put
$v_i(w)=\chi(w-a_i)
$
and
\[
H(w)=\sum_{i=1}^4v_i(w)
+\sum_{|I|=3}\prod_{i\in I}v_i(w),
\]
where $I$ runs over the three element subsets of $\{1,2,3,4\}$.
If $w\notin\{a_1,a_2,a_3,a_4\}$, then all $v_i(w) = \pm 1$ and
\[
2H(w)=\prod_{i=1}^4(1+v_i(w))
-\prod_{i=1}^4(1-v_i(w))\in16\Z.
\]

At $w=a_j$, we have $v_j(a_j)=0$. Hence, in the formula for $H(w)$ every triple product  containing
$j$ vanishes, while the unique triple product omitting $j$ is
\[
\prod_{i\ne j}\chi(a_j-a_i)=p_j.
\]
Therefore
\[
H(a_j)=\sum_{i\ne j}\chi(a_j-a_i)+p_j.
\]
Summing over $j$, the linear terms cancel in pairs since
$\chi(a_j-a_i)=-\chi(a_i-a_j)$, and thus
\[
\sum_{j=1}^4H(a_j)=\sum_{j=1}^4p_j.
\]

By~\eqref{eq:singleton-sum}, the linear terms in $H(w)$ sum to zero over
$k^{\times2}$. Therefore, summing~\eqref{eq:trace-square-sum} over the
four three element subsets $A$ gives
\[
\sum_{|A|=3}T(A)
=-2\sum_{w\in k^{\times2}}H(w)
\equiv-2\sum_{j=1}^4p_j\pmod{16},
\]
as required.
\end{proof}

\begin{lemma}\label{lem:quartet}
Suppose that $\mu_r\subset k$, that $f$ is odd, and that every
$1+\zeta_r^j$, $j\not\equiv0\pmod r$, is a nonzero square in~$k$. Then there exist
integers $1\leq x_1<x_2<x_3<x_4\leq r-1$
such that 
\begin{enumerate}
    \item every subset of $\{x_1,x_2,x_3,x_4\}$ with three elements satisfies~\eqref{eq:triple-conditions};
    \item the four elements $a_i=\omega_{x_i}+1 \in k \setminus \{0\}$ are distinct squares and $p_1=p_2=p_3=p_4=-1$.
\end{enumerate}
\end{lemma}
\begin{proof}
Put $\zeta=\zeta_r\in k$ and define
\[
 \beta_j=\frac{\chi(1-\zeta^j)}{\chi(1-\zeta)}.
\]
Since $f$ is odd, $\chi(-1)=-1$. Moreover, every element of $\mu_r$ and
every $1+\zeta^j$ is a square. Hence
\begin{equation}\label{eq:beta-relations}
 \beta_1=1,\qquad
 \beta_{-j}=-\beta_j,\qquad
 \beta_{2j}=\beta_{3j}=\beta_j,
\end{equation}
because
$1-\zeta^{-j}=-\zeta^{-j}(1-\zeta^j)$,
$1-\zeta^{2j}=(1-\zeta^j)(1+\zeta^j)$, and, in characteristic~$3$,
$1-\zeta^{3j}=(1-\zeta^j)^3$.

Recall that $n=(r-1)/2$. Since $2n\equiv-1\pmod r$, the relations 
\eqref{eq:beta-relations} give
$\beta_n=\beta_{2n}=\beta_{-1}=-1$. Let $b\leq n$ be the least positive
integer such that $\beta_b=-1$. We have $b\geq5$ because
\[
 \beta_1=\beta_2=\beta_3=\beta_4=\beta_6=1.
\]
Also, if $2 \mid b$, 
then $\beta_b=\beta_{b/2}$ contradicts
its minimality. Thus $b$ is odd.
We define
\[
 x_1=b-4,\qquad x_2=b-2,\qquad x_3=b,\qquad x_4=b+2.
\]
By \Cref{lem:valid-triples}, every three-element subset satisfies
\eqref{eq:triple-conditions}, proving part (1). 

Since $\mu_r\subset k$, we have
$3^f\equiv1\pmod r$. Thus \Cref{lem:cyclotomic-squares,lem:omega-nonvanishing}
show that $a_i=\omega_{x_i}+1$ are nonzero squares in~$k$, and
they are distinct by~\eqref{eq:triple-conditions}. We are left to show that $p_j=-1$. 

For the pairs
\[
 (i,j)=(1,2),(1,3),(1,4),(2,3),(2,4),(3,4),
\]
the differences $x_j-x_i$ are respectively $2,4,6,2,4,2$, and hence have
$\beta_{x_j-x_i}=1$. The $x_i$ have the same parity, so $x_i+x_j$ is even, and
the corresponding integers $(x_i+x_j)/2$ are
\[
 b-3,\quad b-2,\quad b-1,\quad b-1,\quad b,\quad b+1.
\]
By the minimality of~$b$, we have $\beta_{b-3}=\beta_{b-2}=\beta_{b-1}=1$. Since $b$ is odd,
$\beta_{b+1}=\beta_{(b+1)/2}=1$ again by minimality of~$b$.
For $i<j$ we have
\[
 \omega_{x_i}-\omega_{x_j}
 =\zeta^{-x_i}(1-\zeta^{x_i+x_j})(1-\zeta^{x_i-x_j}).
\]
The factor $\zeta^{-x_i}$ is a square, therefore
\[
 \chi(\omega_{x_i}-\omega_{x_j})
 =\beta_{x_i+x_j}\beta_{x_i-x_j}
 =\beta_{(x_i+x_j)/2}\beta_{x_i-x_j},
\]
where the second equality follows from~\eqref{eq:beta-relations}. Since
$\beta_{x_j-x_i}=1$ and $\beta_{-j}=-\beta_j$, we have
$\beta_{x_i-x_j}=-1$ for all six pairs. Thus the values
$\chi(\omega_{x_i}-\omega_{x_j})$ corresponding to the pairs in the order above are $-1,-1,-1,-1,1,-1$.

As $a_i-a_j=\omega_{x_i}-\omega_{x_j}$ and
$\chi(a_j-a_i)=-\chi(a_i-a_j)$, it follows that
\[
\begin{aligned}
 p_1&=(-1)(-1)(-1)=-1,&
 p_2&=(1)(-1)(1)=-1,\\
 p_3&=(1)(1)(-1)=-1,&
 p_4&=(1)(-1)(1)=-1.
\end{aligned}
\]
\end{proof}

\begin{theorem}\label{thm:trace-separation}
Let $r\geq11$ be prime, and let $\cT_r$ be the finite set of triples
satisfying~\eqref{eq:triple-conditions}. Then, for every
$[v:w]\in\PP^1(\F_3)$ different from $[1:-1]$ and every prime $\fq\mid3$
of $K$, there exists $\bk\in\cT_r$ such that
\[
 a_\fq(E_{v,w}^\bk)\not\equiv a_\fq(E_{1,-1}^\bk)\pmod{48}.
\]
\end{theorem}
\begin{proof}
Fix $[v:w]\neq[1:-1]$ and a prime $\fq\mid3$. If $v=0$, then
$E_{v,w}^\bk=E_\infty^\bk=E_0^\bk$. If $v\neq0$, put $t=w/v$. Thus it is
enough to treat $t=0$ and $t=1$.

For $t=1$, the result follows from \Cref{prop:t1-direct}. Suppose now that
for $t=0$ every $\bk\in\cT_r$ gives congruent traces modulo~$48$ at~$\fq$.
In particular, the triples $\mathbf f_j$, $j\not\equiv0\pmod r$, and
$(2,6,10)$ give congruent traces modulo~$3$, so
\Cref{lem:direct-preparation} applies. Thus $\mu_r\subset k$, $f$ is odd,
and every $1+\zeta_r^j$, $j\not\equiv0\pmod r$, is a nonzero square.

By \Cref{lem:quartet}, choose $x_1,x_2,x_3,x_4$ and put
$a_i=\omega_{x_i}+1$. Every subset of $\{x_1,x_2,x_3,x_4\}$ with three elements satisfies
\eqref{eq:triple-conditions}, the four $a_i$ are distinct nonzero squares,
and $p_j=-1$ for all~$j$.

Apply \Cref{lem:four-point-count} to $a_1,a_2,a_3,a_4$. If
$A=\{a_{i_1},a_{i_2},a_{i_3}\}$, let
$\bk=(x_{i_1},x_{i_2},x_{i_3})$, reordered increasingly. Then
$A=A_\bk$, and \Cref{lem:t0-models} gives
\[
 T(A)=a_\fq(E_{-1}^\bk)-a_\fq(E_0^\bk).
\]
By assumption each of these four differences is divisible by~$48$, hence
by~$16$. On the other hand, \Cref{lem:four-point-count} gives
\[
 \sum_{|A|=3}T(A)\equiv-2\sum_{j=1}^4p_j=8\pmod{16},
\]
a contradiction.
\end{proof}

\section{Proof of \Cref{thm:main}}

Let $r \geq 11$ be a prime.
By \Cref{thm:trace-separation}, the finite set $\cT_r$ satisfies the
hypothesis of \Cref{thm:local-criterion}. Hence \Cref{thm:main} follows.

\end{document}